\documentclass[12pt]{amsart}

\usepackage{fullpage, parskip}
\usepackage{amsmath, amsthm, amsfonts}
\usepackage{graphicx}
\newtheorem{theorem}{Theorem}[section]
\newtheorem{corollary}[theorem]{Corollary}
\newtheorem{lemma}[theorem]{Lemma}

\title{Chromatic Polynomials of $2$-edge-coloured Graphs}
\author{Iain Beaton$^1$}
\address{$^1$Department of Mathematics and Statistics, Dalhousie University, Canada}
\author{Danielle Cox$^2$}
\address{$^2$Mathematics Department, Mount Saint Vincent University, Canada; Research supported by the Natural Science and Engineering Council of Canada}
\author{Christopher Duffy$^3$}
\address{$^3$Department of Mathematics and Statistics, University of Saskatchewan, Canada; Research supported by the Natural Science and Engineering Council of Canada}
\author{Nicole Zolkavich$^4$}
\address{$^4$Department of Mathematics and Statistics, McGill University, Canada}

\begin{document}

\begin{abstract}
	Using the definition of colouring of $2$-edge-coloured graphs derived from $2$-edge-coloured graph homomorphism, we extend the definition of chromatic polynomial to $2$-edge-coloured graphs. 
	We find closed forms for the first three coefficients of this polynomial that generalize the known results for the chromatic polynomial of a graph.
	We classify those $2$-edge-coloured graphs that have a chromatic polynomial equal to the chromatic polynomial of the underlying graph, when every vertex is incident to edges of both colours.
	Finally, we examine the behaviour of the roots of this polynomial, highlighting behaviours not seen in chromatic polynomials of graphs.
\end{abstract}

\maketitle

\section{Introduction and Preliminary Notions}

A \emph{$2$-edge-coloured graph $G$} is a triple $(\Gamma,R_G,B_G)$ where $\Gamma$ is a simple graph,  $R_G \subseteq E(\Gamma)$, and $B_G \subseteq E(\Gamma)$ so that $R_G \cap B_G = \emptyset$  and $R_G\cup B_G = E(\Gamma)$.
We call $G$ a \emph{$2$-edge-colouring} of $\Gamma$.
We let $G[R_G]$ and $G[B_G]$ be the simple graph induced, respectively, by the set of red and blue edges.
We call $\Gamma$ the simple graph \emph{underlying} $G$.
%We call  $G =(\Gamma,R_G,B_G)$ a \emph{$2$-edge-colouring of $\Gamma$}.
We note  $\{R_G,B_G\}$ need not be a partition of $E(\Gamma)$; we permit $R_G =\emptyset$ or $B_G =\emptyset$.
In the case that $\{R_G,B_G\}$ is a partition of $E(\Gamma)$, we say that $G$ is \emph{bichromatic}.
Otherwise we say $G$ is \emph{monochromatic.}
When there is no chance for confusion we refer to $R_G$ and $B_G$ as $R$ and $B$, respectively.
Herein we assume all graphs are loopless and have no parallel arcs.
Thus we drop the descriptor of simple when we refer to simple graphs.
In various corners of the literature \cite{moprs08,S14} $2$-edge-coloured graphs are referred to as \emph{signified graphs} to highlight the absence of a switching operation as in \cite{N15}.

For other notation not defined herein, we refer the reader to \cite{bondy}.
Throughout we generally use Greek majuscules to refer to graphs and Latin majuscules to refer to $2$-edge-coloured graphs.

Let  $G=\left(\Gamma_G,R_G,B_G\right)$ and $H=\left(\Gamma_H,R_H,B_H\right)$ be $2$-edge-coloured graphs.
There is a \emph{homomorphism} of $G$  to $H$  when there exists a homomorphism $\phi:\Gamma_G \to \Gamma_H$ so that $\phi: G[R_G] \to H[R_H]$ and $\phi: G[B_G] \to H[B_H]$ are both graph homomorphisms. 
That is, a homomorphism of $G$ to $H$ is a vertex mapping from $V(\Gamma_G)$ to $V(\Gamma_H)$ that preserves the existence of edges as well as their colours. 
When $H$ has $k$ vertices we call $\phi$ a \emph{$k$}-colouring of $G$.

Equivalently, one may view a $k$-colouring of a $2$-edge-coloured graph $G= (\Gamma,R,B)$ as a function $c: V(G) \to \{1,2,3,\dots, k\}$ satisfying the following two conditions
\begin{enumerate}
	\item for all $yz \in E(\Gamma)$, we have $c(y) \neq c(z)$; and
	\item  for all $ux\in R$ and $vy \in B$, if $c(u) = c(v)$ then $c(x) \neq c(y)$.
\end{enumerate}

Such a proper colouring defines a homomorphism to  $2$-edge-coloured target $H$ with vertex set $\{1,2,3, \dots, k\}$ where $ij \in R_H$ (respectively, $\in B_H$) if and only if there is an edge $wx \in R_G$ (respectively, $\in B_G$) so that $c(w) = i$ and $c(x) = j$.

The \emph{chromatic number of $G$}, denoted $\chi(G)$, is the least integer $t$ such that $G$ admits a $t$-colouring.
Observe that when $G$ is monochromatic, the definitions above are equivalent to those for graph homomorphism, $k$-colouring and chromatic number.
Thus the choice of notation, $\chi$, to denote the chromatic number is indeed appropriate.

We note that for a $2$-edge-coloured graph $G$, the chromatic number of $G = \left(\Gamma,R,B\right)$ may differ vastly from that of $\Gamma$.
There exist $2$-edge-colourings of bipartite graphs that have chromatic number equal to their number of vertices \cite{S14}.

Homomorphisms of edge-coloured graphs have received increasing attention in the literature in the last twenty-five years.
Early work by Alon and Marshall \cite{AM98} gave an upper bound on the chromatic number of a $2$-edge-coloured planar graph.
More recent work has bounded the chromatic number of two-edge-coloured graphs from a variety of graph families \cite{moprs08,O17} as well as considered questions of computational complexity \cite{B17,B00}.

Given the definition of colouring given above as proper colouring of the underlying graph satisfying extra constraints, one can enumerate, for fixed positive integer $k$, the number of $k$-colourings of a $2$-edge-coloured graph $G$.
This observation gives rise to a reasonable definition of chromatic polynomial for a $2$-edge-coloured graph.
As the definition of homomorphism and colouring of $2$-edge-coloured graphs fully captures those for graphs, we notice  results for the chromatic polynomial of $2$-edge-coloured graphs necessarily generalize those for graphs.

To study the chromatic polynomial for $2$-edge-coloured graphs, we introduce a chromatic polynomial of a more generalized graph object.
A \emph{mixed $2$-edge-coloured graph} is a pair $M = (G,F_M$) where $G$ is  $2$-edge-coloured graph where $G = (\Gamma,R_G,B_G)$ and   $F_M \subseteq \overline{E(\Gamma)}$.
We consider $F_M$ as a set of edges that are contained neither in $R_G$ nor in $B_G$.
We denote by $S(M)$ the graph with vertex set $V(\Gamma)$ and edge set $R_G \cup B_G \cup F_M$.
When there is a no chance for confusion we refer to $F_M$ as $F$.

At this point it may be tempting to consider a mixed $2$-edge-coloured graph as a $3$-edge-coloured graph.
We resist this temptation as the following definition of $k$-colouring of a mixed $2$-edge-coloured graph does not match what one would expect for a $3$-edge-coloured graph.

Let $M = (G,F)$ be a mixed $2$-edge-coloured graph with $G = \left(\Gamma, R, B\right)$
We define a \emph{$k$-colouring} of $M$ to be a function $c:V(G) \to \{1,2,3, \dots , k\}$ so that
\begin{enumerate}
	\item $c(u) \neq c(v)$ for all $uv \in R \cup B \cup F$; and 
	\item for all $ux\in R$ and  for all $vy \in B$, if $c(u) = c(v)$ then $c(x) \neq c(y)$.
\end{enumerate}

%As with $k$-colouring of a $2$-edge-coloured graph, the definition of $k$-colouring of  a mixed $2$-edge-coloured graph is equivalent to one arising from a reasonable definition of homomorphism of mixed $2$-edge-coloured graphs.

Informally, $c$ is a $k$-colouring of the $2$-edge-coloured graph $G$ with the extra condition that vertices at the ends of an element of $F$ receive different colours.
Notice when $M = \emptyset$, we have that $c$ is a $k$-colouring of $G$.
Further notice when $R = B = \emptyset$ we have that $c$ is a $k$-colouring of the graph with vertex set $V(\Gamma)$ and edge set $F$.

Let $M$ be a mixed $2$-edge-coloured graph and let $P(M,\lambda)$ be the unique interpolating polynomial so that for every non-negative integer $k$, we have that $P(M,k)$ is the number of $k$-colourings of $M$.
We refer to $P(M,\lambda)$ as the chromatic polynomial of $M$.
Notice  when $F_M = \emptyset$, the polynomial $P(M,\lambda)$ enumerates colourings of a $2$-edge-coloured graph and when $R = B = \emptyset$, this polynomial is identically the chromatic polynomial of the graph with vertex set $V(\Gamma)$ and edge set $F$.

Our work proceed as follows.
In Section \ref{sec:chromPoly} we study properties of chromatic polynomials of mixed $2$-edge-coloured graphs.
We use a recurrence reminiscent of the standard recurrence for the chromatic polynomial of a graph to give closed forms for the first three coefficients of this polynomial.
These results generalize known results for chromatic polynomials of graphs.
As every colouring of a $2$-edge-coloured graph is a colouring of the underlying graph, we are led naturally to considering those $2$-edge-coloured graphs that have the same chromatic polynomial as their underlying graph.
We study this problem in Section \ref{sec:chromInvar}.
For $2$-edge-coloured graphs where each vertex is incident with both a red edge and a blue edge,  we find  those $2$-edge-coloured graphs that have the same chromatic polynomial as their underlying graph to be those $2$-edge-coloured graphs that arise as a $2$-edge-colouring of a join.
In Section \ref{sec:roots} we study roots of chromatic polynomials of $2$-edge-coloured graphs.
We find that closure of the real roots of the $2$-edge-coloured chromatic polynomials to be $\mathbb{R}$ and that the non-real roots can have arbitrarily large modulus.
Finally in Section \ref{sec:discuss} we compare our results for chromatic polynomials with those for graphs and oriented graphs.
We conjecture on the properties of the chromatic polynomial of an $(m,n)$-mixed graph.

\section{The Chromatic Polynomial of a mixed $2$-edge-coloured Graph}\label{sec:chromPoly}

Our study of the chromatic polynomial of a mixed $2$-edge-coloured graph begins by examining two cases for (2) in the definition of a $k$-colouring of a mixed $2$-edge-coloured graph.

Consider first the case $u=v$.
In this case this condition enforces that vertices at the end of an induced bichromatic path with three vertices receive different colours in an colouring.
Let $M = \left(G,F\right)$ be a mixed $2$-edge-coloured graph.
An induced path $xuy$ is called a \emph{bichromatic $2$-path} when $xu \in R$ and $uy \in B$.
Let $\mathcal{P}_G$ be the set of bichromatic $2$-paths of $G$.

Let $M = (G,F)$ be a mixed $2$-edge-coloured graph with $n$ vertices.
If every pair of vertices of $M$ is either adjacent in $M$ or at the ends of a bichromatic $2$-path in $G$, then in any colouring of $M$ each vertex receives a distinct colour.
Thus

\begin{equation} \label{math:complete}
 P(M,\lambda) =  \Pi_{i =0}^{n-1} \left(\lambda-i\right) 
\end{equation}

As every vertex of $M$ receives a distinct colour in every colouring of $M$, the chromatic polynomial of $M$ is exactly that of $K_n$.
We return to this observation in the proofs of Theorems \ref{thm:secondcoeff} and \ref{thm:thirdcoeff}.

Let $x$ and  $y$ be a pair of vertices that are neither adjacent in $M$ nor at the ends of a bichromatic $2$-path in $G$.
The $k$-colourings of $M$ can be partitioned into those in which $x$ and $y$ receive the same colour and those in which $u$ and $v$ receive different colours.
Thus

\begin{equation} \label{math:reduce}
P(M,\lambda) =  P(M+xy,\lambda) +  P(M_{xy},\lambda)
\end{equation}

where
\begin{itemize}
	\item $M+xy$ is the mixed $2$-edge-coloured graph formed from $M$ by adding $xy$ to $F$; and
	\item $M_{xy}$ is the mixed $2$-edge-coloured graph formed from identifying vertices $x$ and $y$ and deleting any edge that is parallel with a coloured edge.
\end{itemize}

See Figure \ref{fig:exampleFig} for a sample computation.
Elements of $R$ and $B$ are denoted respectively by dotted and dashed lines.
Elements of $F$ are denoted by solid lines.
Here we follow the usual convention for chromatic polynomials of having the picture of the graph stand in for its polynomial.

\begin{figure}
	\begin{center}
		\includegraphics[width=0.75\linewidth]{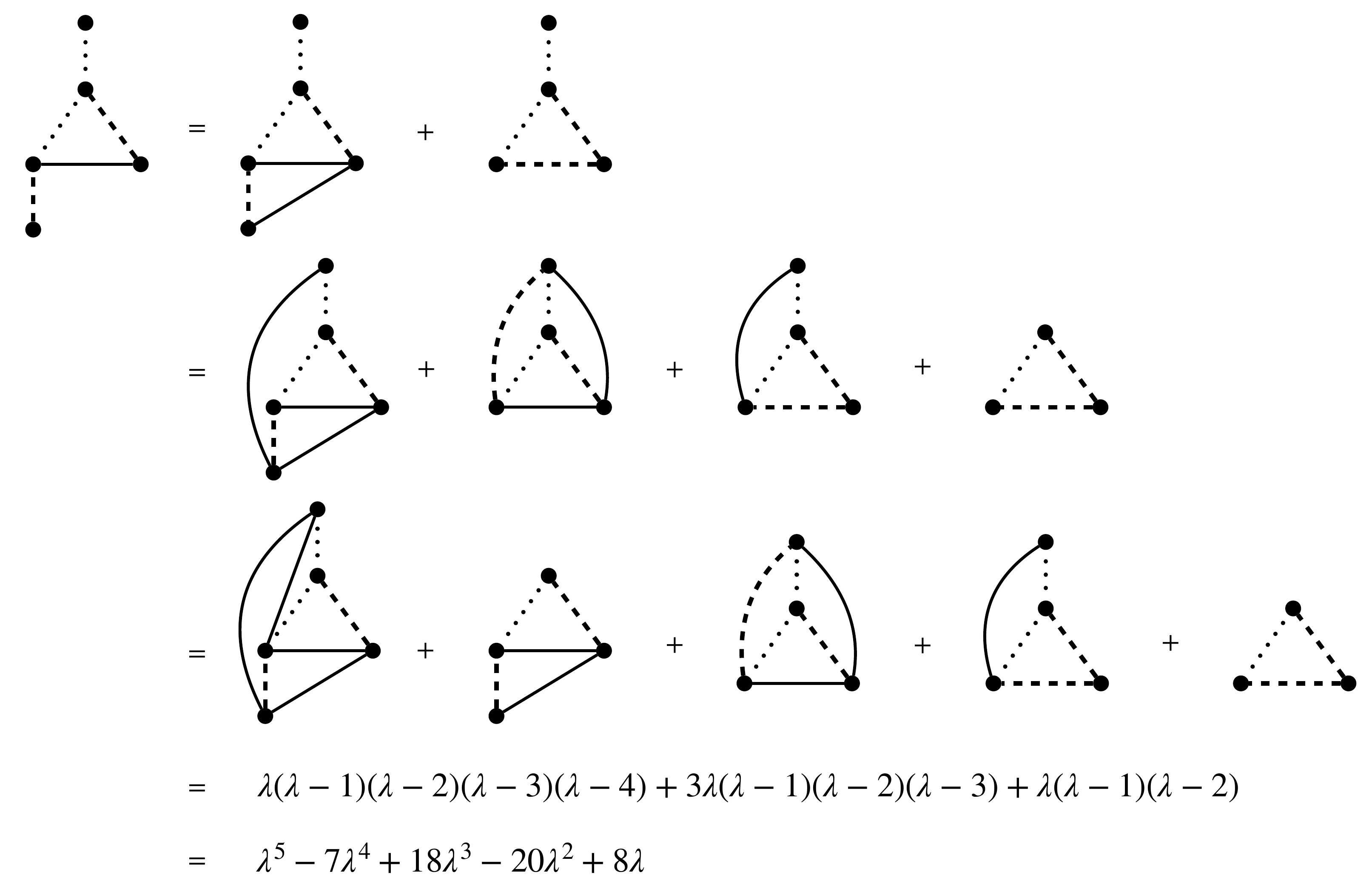}
		\caption{Computing the chromatic polynomial of a mixed $2$-edge-coloured graph}
		\label{fig:exampleFig}
	\end{center}
\end{figure}

From Equations (\ref{math:complete}) and (\ref{math:reduce}) we directly obtain the following.

\begin{theorem}\label{thm:basicResults}
	Let $M = (G,F)$ be a mixed graph with $n$ vertices.
	\begin{enumerate}
		\item  $P(M,\lambda)$ is a polynomial of degree $n$ in $\lambda$;
		\item the coefficient of $\lambda^n$ is $1$; and
		\item  $P(M,\lambda)$ has no constant term.
	\end{enumerate}
\end{theorem}

As with chromatic polynomials, the coefficient of $\lambda^{n-1}$ is easily computed.

\begin{theorem}\label{thm:secondcoeff}
	For a mixed graph $M = (G,F)$ the coefficient of $\lambda^{n-1}$ is given by \[-\left(|R_G| + |B_G| + |F_M| + |\mathcal{P}_G| \right).\]
\end{theorem}

\begin{proof}
	We proceed considering the existence of a counterexample.
	Let $n$ be the least integer so that there exists a mixed $2$-edge-coloured graph with $n$ vertices so that the statement of the theorem is false.
	Among all counterexamples on $n$ vertices, let $M = (G,F)$ be a counterexample that maximizes the number of edges in $S(M)$.
	If $\mathcal{P}_G \neq \emptyset$, then there exists vertices $u$ and $v$ which are the ends of an induced bichromatic path in $M$.
	Let $M^\prime$ be the mixed $2$-edge-coloured graph formed from $M$ by adding $uv$ to $F$.
	We observe that every colouring of $M$ is a colouring of $M^\prime$ and also that every colouring of $M^\prime$ is a colouring of $M$.
	Thus $M$ and $M^\prime$ have the same chromatic polynomial.
	This contradicts our choice of $M$ as a counterexample on $n$ vertices with the maximum number of edges in $S(M)$.
	Thus we may assume $\mathcal{P}_G = \emptyset$.
	
	We claim there exists a pair of vertices $u$ and $v$ so that $u$ and $v$ are not adjacent in $S(M)$.
	If $u$ and $v$ do not exist, then $S(M) \cong K_n$ and so $P(M,\lambda) = P(K_n,\lambda)$.
	The coefficient of $\lambda^{n-1}$ in $P(K_n,\lambda)$ is given by $-{n \choose 2}$ \cite{R68}.
	Recall  $\mathcal{P}_G = \emptyset$.
	For $M$ we observe
	
	\[
		-\left(|R_G| + |B_G| + |F_M| + |\mathcal{P}_G| \right)	=	-{n \choose 2}
	\]
	
	Thus we consider $u,v \in V(M)$ so that $u$ and $v$ are not adjacent in $S(M)$.
	Let $c_{n-1}$ be the coefficient of $\lambda^{n-1}$ in $P(M,\lambda)$.
	By Theorem \ref{thm:basicResults}, our choice of $M$, and Equation (\ref{math:reduce}), we have 
	
	\begin{align*}
	c_{n-1} &= -\left(|R_{G}| + |B_{G}| + |F_{M+uv}| + |\mathcal{P}_{G}| \right) +1
	\end{align*}
	
	We observe  $|F_{M+uv}| = |F_{M}| + 1$.
	Simplifying yields 
	
	\begin{align*}
	c_{n-2} = 	-\left(|R_G| + |B_G| + |F_M| + |\mathcal{P}_G| \right)
	\end{align*}
Thus $M$ is not a counterexample.
And so by choice of $M$, no counterexample exists.
\end{proof}

\begin{corollary}
	For a $2$-edge-coloured graph $G$, the coefficient of $\lambda^{n-1}$ in $P(G,\lambda)$ is given by $-\left(|R_G| + |B_G| + |\mathcal{P}_G|\right)$
\end{corollary}

We return now to our examination of cases for (2) in the definition of $k$-colouring of a mixed $2$-edge-coloured graph $M = (G,F)$.
Consider now the case where $ux$ and $vy$ induce a bichromatic copy of $2K_2$ in $G$.
For $c$, a $k$-colouring of $M$, we note that $\left|\{c(u),c(x),c(v),c(z)\}\right| \in \{3,4\}$.
As the subgraph induced by $\{u,v,y,z\}$ has chromatic number $2$ in $S(M)$, such pairs of edges, in a sense, obstruct a colouring of $S(M)$ from being a colouring of $M$.

Let $M = (G,F)$ be a mixed $2$-edge-coloured  graph.
Let $\Lambda$ be the graph formed from $M$ by adding to $F$ the edge between any pair of vertices which are the ends of an induced bichromatic $2$-path in $M$ and considering coloured edges in $G$ as edges in $\Lambda$.
(See Figure \ref{fig:MMstar} for an example)

\begin{figure}
	\begin{center}
		\includegraphics[width=0.25\linewidth]{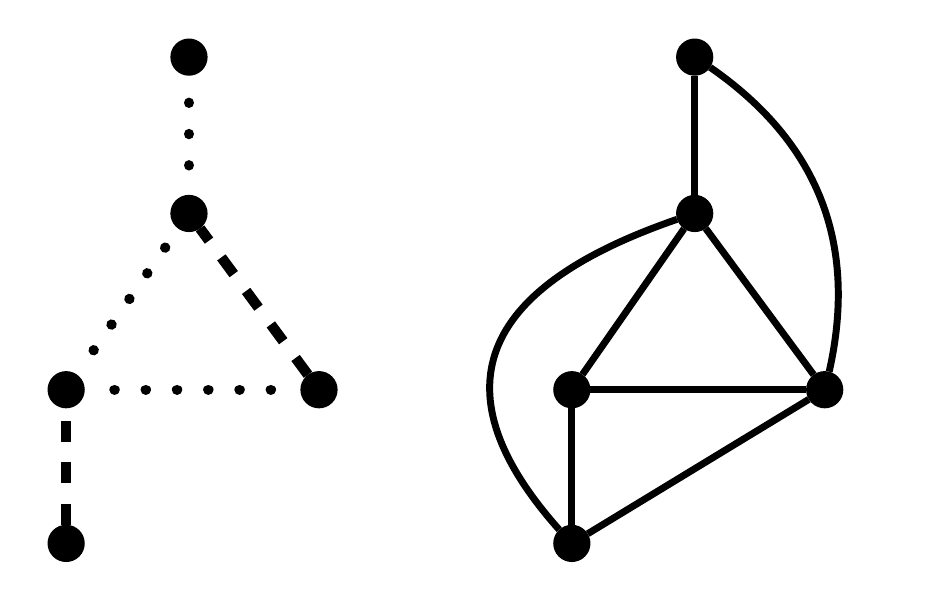}
		\caption{$\Lambda$ (right) constructed from $M$ (left)}
		\label{fig:MMstar}
	\end{center}
\end{figure}

For $ux \in R$ and $vy \in B$ we say that $ux$ and $vy$ are \emph{obstructing edges} when $\chi(\Lambda[u,x,v,y]) = 2$.
Let $\mathcal{O}_M$ be the set of pairs of obstructing edges in $M$. 
For a $2$-edge-coloured graph $G$, we let  $\mathcal{O}_G$ be the set of obstructing edges in $G$. 

Using the notion of obstructing edges, we find a closed form for the coefficient of $\lambda^{n-2}$ of a chromatic polynomial of a mixed $2$-edge-coloured graph.

\begin{theorem}\label{thm:thirdcoeff}
	For $M = (G,F)$ a mixed $2$-edge-coloured  graph, the coefficient of $\lambda^{n-2}$ in $P(M,\lambda)$ is given by
	\[
{	{|R_G| + |B_G| + |\mathcal{P}_G|} + |F_M|\choose 2}
		  - |T_{S(M)}| - |\mathcal{P}_G| - |\mathcal{O}_M|,
	\]
 	where $T_{S(M)}$ is the set of induced subgraphs of $S(M)$ isomorphic to $K_3$.
\end{theorem}

\begin{proof}
	We proceed considering the existence of a counterexample.
	Let $n$ be the least integer so that there exists a mixed $2$-edge-coloured graph with $n$ vertices so that the statement of the theorem is false.
	Among all counterexamples on $n$ vertices, let $M = (G,F)$ be a counterexample that maximizes the number of edges in $S(M)$.

	If $\mathcal{P}_G \neq \emptyset$, then there exists vertices $u$ and $v$ which are the ends of an induced bichromatic path in $M$.
	Let $M^\prime$ be the mixed $2$-edge-coloured graph formed from $M$ by adding $uv$ to $F$.
	We observe that every colouring of $M$ is a colouring of $M^\prime$ and also that every colouring of $M^\prime$ is a colouring of $M$.
	Thus $M$ and $M^\prime$ have the same chromatic polynomial.
	This contradicts our choice of $M$.
	Thus we may assume $\mathcal{P}_G = \emptyset$.
	
	We claim there exists a pair of vertices $u$ and $v$ so that $u$ and $v$ are not adjacent in $S(M)$.
	If $u$ and $v$ do not exist, then $S(M) \cong K_n$ and so  $P(M,\lambda) = P(K_n,\lambda)$.
	The coefficient of $\lambda^{n-2}$ in $P(K_n,\lambda)$ is given by$ {{n \choose 2} \choose 2} - {n \choose 3}$ \cite{R68}.
	For $M$ we observe
	\begin{align*}
	{{n \choose 2} \choose 2} - {{n \choose 3}} &= 	{|R_G| + |B_G| + |\mathcal{P}_G| + |F_M|\choose 2} - |T_{S(M)}|  \\
	&= 	{|R_G| + |B_G| + |\mathcal{P}_G| + |F_M|\choose 2} - |T_{S(M)}| - |\mathcal{P}_G| - |\mathcal{O}_M|
	\end{align*}

This last equality follows by observing  $\mathcal{P}_G=\emptyset$ (by hypothesis) and $\mathcal{O}_M = \emptyset$ when $M$ is complete.
This equality contradicts our choice of $M$, and so we conclude that such vertices $u$ and $v$ exist.

	Thus we consider $u,v \in V(M)$ so that $u$ and $v$ are not adjacent in $S(M)$.
	Let $c_{n-2}$ be the coefficient of $\lambda^{n-2}$ in $P(M,\lambda)$.
	By Theorem \ref{thm:secondcoeff}, our choice of $M$ and Equation (\ref{math:reduce})  we have 
	\begin{align*}
c_{n-2} &={|R_{G}| + |B_{G}| +	|\mathcal{P}_{G}| + |F_{M+uv}| \choose 2} 
- \left(|T_{S(M+uv)}| + |\mathcal{P}_{G}| + |\mathcal{O}_{M+uv}|\right)\\
&-\left(|R_{G_{uv}}| + |B_{G_{uv}}| + |F_{M_{uv}}| + |\mathcal{P}_{G_{uv}}| \right)
	\end{align*}
	
Observe  $|F_{M+uv}| = |F_{M}| + 1$.
Let $C$ be the set of common neighbours of $u$ and $v$ in $S(M)$.
We observe  $|T_{S(M+uv)}| = |T_{S(M)}| + |C|$ and $|R_{G_{uv}}| + |B_{G_{uv}}| + |F_{M_{uv}}| = |R_{G}| + |B_{G}| + |F_{M}| - |C|$.
Thus 

	\begin{align*}
c_{n-2} &={|R_{G}| + |B_{G}| +	|\mathcal{P}_{G}| + |F_{M}| + 1 \choose 2} 
- \left(|T_{S(M)}| + |\mathcal{P}_{G}| + |\mathcal{O}_{M+uv}|\right)\\
&-\left(|R_{G}| + |B_{G}| + |F_{M}| + |\mathcal{P}_{G_{uv}}| \right)
\end{align*}
	
Since $\mathcal{P}_{G} = \emptyset$, a pair of obstructing edges, $\left(ux, vy\right)$ in $M$, is not obstructing in $M+uv$ if and only if $xy \notin F$ and  one of $uy$ or $vx$ is contained in $F$.
Let $\mathcal{O}^{uv}_{M}$ be the set of such obstructing edges.
Therefore $|\mathcal{O}_{M+uv}| = |\mathcal{O}_{M}| - |\mathcal{O}^{uv}_{M}|$.
Notice now that every element of $\mathcal{O}^{uv}_{M}$ contributes an element of $\mathcal{P}_{G_{uv}}$ that was not an element of $\mathcal{P}_{G}$. Thus $|\mathcal{P}_{G_uv}| =  |\mathcal{P}_{G}| +|\mathcal{O}^{uv}_{M}|$.
	
Substituting yields

\begin{align*}
c_{n-2} &={|R_{G}| + |B_{G}| +	|\mathcal{P}_{G}| + |F_{M}| + 1 \choose 2} 
- \left(|T_{S(M)}| + |\mathcal{P}_{G}| + \mathcal{O}_{M} - |\mathcal{O}^{uv}_{M}|\right)\\
&-\left(|R_{G}| + |B_{G}| + |F_{M}| + |\mathcal{P}_{G}| +|\mathcal{O}^{uv}_{M}| \right)\\
& ={|R_{G}| + |B_{G}| +	|\mathcal{P}_{G}| + |F_{M}| + 1 \choose 2}  -\left(|R_{G}| + |B_{G}| + |F_{M}| + |\mathcal{P}_{G}| \right) - \left(|T_{S(M)}| + |\mathcal{P}_{G}| + |\mathcal{O}_{M}|\right)\\
& ={|R_{G}| + |B_{G}| +	|\mathcal{P}_{G}| + |F_{M}|\choose 2}  - \left(|T_{S(M)}| + |\mathcal{P}_{G}| + |\mathcal{O}_{M}|\right)
\end{align*}
	
Thus $M$ is not a counterexample.
And so by choice of $M$, no counterexample exists.
\end{proof}

\begin{corollary}
	For a $2$-edge-coloured graph $G=\left(\Gamma,R,B\right)$, the coefficient of $\lambda^{n-2}$ in $P(G,\lambda)$ is given by
	\[
		{|R_G| + |B_G| + |\mathcal{P}_G|\choose 2} 
- |T_{\Gamma}| - |\mathcal{P}_G| - |\mathcal{O}_G|
\]
\end{corollary}

The results in Theorems \ref{thm:secondcoeff} and \ref{thm:thirdcoeff} relied on known results about the second and third coefficients of the chromatic polynomial of a complete graph.
Such results can be proved independently for mixed $2$-edge-coloured complete graphs, forgoing the need to reference prior work on the chromatic polynomial of a graph.
In this case we notice that the results for the second and third coefficients of the chromatic polynomial may be obtained as corollaries of Theorems \ref{thm:secondcoeff} and \ref{thm:thirdcoeff}  by considering mixed $2$-edge-coloured graphs with $R = B = \emptyset$.
We  contextualize our results within the current state-of-the-art in Section \ref{sec:discuss}.

\section{Chromatically Invariant $2$-edge-coloured Graphs}\label{sec:chromInvar}

Consider a $2$-edge-coloured graph $G = \left(\Gamma,R,B\right)$.
By definition every colouring of $G$ is necessarily a colouring of $\Gamma$.
Thus for each integer $k \geq 1$ we have $P(\Gamma,k) \geq P(G,k)$.
In this section we study the structure of $2$-edge-coloured graphs $G$ for which $P(\Gamma,\lambda) = P(G,\lambda)$.
We refer to such $2$-edge-coloured graphs as \emph{chromatically invariant}.
Trivially, every $2$-edge-coloured graph in which $R = \emptyset$ (or $B = \emptyset$) is chromatically invariant.
We refer to those chromatically invariant $2$-edge-coloured graphs with  $R,B \neq \emptyset$ \emph{non-trivially chromatically invariant}.

\begin{lemma}\label{lem:doubleEmpty} 
	Let $G$ be a $2$-edge-coloured graph. If $\mathcal{P}_G = \mathcal{O}_G = \emptyset$, then $G$ is chromatically invariant.
\end{lemma}

\begin{proof}
	Let $G=\left(\Gamma,R,B\right)$ be a $2$-edge-coloured graph so that $\mathcal{P}_G = \mathcal{O}_G = \emptyset$.
	For each $k \geq 1$, let $C_{G,k}$ be the set of $k$-colourings of $G$.
	Similarly, let $C_{\Gamma,k}$ be the set of $k$-colourings of $\Gamma$.
	Recalling the definition of $k$-colouring of a $2$-edge-coloured graph, it follows directly that $C_{G,k} \subseteq C_{\Gamma,k}$.
	To complete the proof it suffices to show  $C_{\Gamma,k} \subseteq C_{G,k}$.
	
	Let $c$ be a $k$-colouring of $\Gamma$.
	Consider $ux\in R$ and $vy \in B$ so that $c(u) = c(v)$.
	If $u = v$, then $xy \in E(\Gamma)$ as $\mathcal{P}_G = \emptyset$.
	Thus $c(x) \neq c(y)$.
	
	Consider now the case where $u \neq v$.
	Since $\mathcal{O}_G = \emptyset$ it follows that $\left|\{c(u),c(x),c(v), c(y)\}\right| \in \{3,4\}$.
	Thus $c(x) \neq c(y)$.
	Therefore $c \in C_{G,k}$ and so it follows $C_{\Gamma,k} \subseteq C_{G,k}$.
\end{proof}

\begin{theorem}\label{thm:firstCharact}
	A $2$-edge-coloured graph $G = \left(\Gamma ,R , B\right)$ is chromatically invariant if and only if $G$ has no induced bichromatic $2$-path and $G$ contains no induced bichromatic copy of $2K_2$.
\end{theorem}

\begin{proof}
	Let $G=\left(\Gamma,R,B\right)$ be a $2$-edge-coloured graph.
	
	Begin by assuming $G$ is chromatically invariant.
	For a contradiction, we first assume $\mathcal{P} \neq \emptyset$.
	Consider $uvw \in \mathcal{P}$.
	Let $k$ be the least integer so that there is a $k$-colouring $c$ of $\Gamma$ for which $c(u) = c(v)$.
	Let $C_{\Gamma,k}$ be the set of $k$-colourings of $\Gamma$.
	Let $C_{G,k}$ be the set of $k$-colourings of $G$.
	By construction, $c \in C_{\Gamma,k}$ but $c \notin C_{G,k}$.
	By the argument in the proof of Lemma \ref{lem:doubleEmpty}, we have $C_{G,k} \subseteq C_{\Gamma,k}$.
	Therefore $|C_{G,k}| < |C_{\Gamma,k}|$.
	Thus $P(G,k) < P(\Gamma, k)$, which implies $P(G,\lambda) \neq P(\Gamma, \lambda)$.
	
	Assume now $G$ contains an induced bichromatic copy of $2K_2$.
	Let $ux \in R$ and $vy \in B$ so that $G[u,x,v,y]$ is a bichromatic copy of $2K_2$.
	Let $k$ be the least integer so that there is a $k$ colouring $c$ of $\Gamma$ for which $c(u) = c(v)$ and $c(x) = c(y)$.
	A contradiction follows as in the previous paragraph.

	Assume $\mathcal{P} = \emptyset$ and $G$ contains no induced bichromatic copy of $2K_2$.
	By Lemma \ref{lem:doubleEmpty} it suffices to show $\mathcal{O}_G = \emptyset$.
	Consider $ux \in R$ and $vy \in B$ so that $u \neq v,y$ and $x \neq v,y$.
	Since $G$ contains no induced bichromatic copy of $2K_2$, there exists an edge with an end in $\{u,x\}$ and an end in $\{v,y\}$.
	Without loss of generality, assume $uv \in R$.
	Since $\mathcal{P} = \emptyset$ and $vy \in B$ it follows that $uy \in E(\Gamma)$.
	Therefore $\Gamma[u,v,x,y]$ contains a copy of $K_3$.
	Thus $\left(ux,vy\right)$ is not a pair of obstructing edges.
	Therefore $\mathcal{O}_G = \emptyset$.
	The result follows by Lemma \ref{lem:doubleEmpty}.	
\end{proof}

Theorem \ref{thm:firstCharact} gives a full characterization of chromatically invariant $2$-edge-coloured graphs. 
This characterization allows us to further characterize these $2$-edge-coloured graphs by way of pairs of independent sets.

\begin{theorem}\label{thm:I1I2}
	A $2$-edge-coloured graph $G=\left(\Gamma,R,B\right)$ is chromatically invariant if and only if for every disjoint pair of non-empty independent sets $I_1$ and $I_2$ in $\Gamma$, the $2$-edge-coloured subgraph induced by $I_1$ and $I_2$ is monochromatic.
\end{theorem}

\begin{proof}
	Let $G=\left(\Gamma,R,B\right)$ be a $2$-edge-coloured graph and let $I_1$ and $I_2$ be disjoint non-empty independent sets of $\Gamma$.
	
	Assume $G$ is chromatically invariant.
	Thus by Theorem \ref{thm:firstCharact}, it follows that $G$ has no induced bichromatic $2$-path and no induced bichromatic copy of $2K_2$.
	If $G[I_1 \cup I_2]$ has at most one edge, then the result holds --  necessarily this subgraph is monochromatic.
	Otherwise, assume $e = u_1u_2$ and $f = v_1v_2$ are edges of $G[I_1 \cup I_2]$ ($u_1,v_1 \in I_1$ and $u_2,v_2 \in I_2$).
	Assume, without loss of generality, that $e \in R$ and $f \in B$.

	We first show that $e$ and $f$ have a common end point.
	Recall $G$ contains no induced bichromatic copy of $2K_2$
	Thus if $e$ and $f$ do not have a common end point, then, without loss of generality, we have $u_1v_2 \in E(\Gamma)$.
	If $u_1v_2 \in R$ then $u_1v_2v_1 \in \mathcal{P}$.
	Similarly, if $u_1v_2 \in B$, then $v_2u_1v_1 \in \mathcal{P}$.
	However, we have $\mathcal{P} = \emptyset$.
	Therefore $e$ and $f$ have a common end point.
	
	If $e$ and $f$ have a common endpoint, then $ef \in \mathcal{P}$.
	However, we have $\mathcal{P} = \emptyset$.
	Therefore $e$ and $f$ do not have a common end point, a contradiction.
	Therefore the $2$-edge-coloured subgraph induced by $I_1$ and $I_2$ is monochromatic.
	
	Assume $G$ is not chromatically invariant.
	By Theorem \ref{thm:firstCharact} we have that $P \neq \emptyset$ or $G$ contains an induced bichromatic copy of $2K_2$.
	In either case we find a pair of disjoint independent sets $I_1,I_2$ so that $G[I_1 \cup I_2]$ is not monochromatic.	
\end{proof}

\begin{corollary}\label{cor:subgraphs}
	If $G$ is a $2$-edge-coloured chromatically invariant graph, then every induced subgraph of $G$ is chromatically invariant.
\end{corollary}

We turn now to the problem of classifying those graphs which admit a chromatically invariant $2$-edge-coloured colouring.
We focus our attention on those graphs which admit a chromatically invariant $2$-edge-colouring in which every vertex is incident with both a red and blue edge.

Recall that a graph $\Gamma$ is a \emph{join} when  $V(\Gamma)$ has a partition $\{X,Y\}$ so that $xy \in E(\Gamma)$ for all $x\in X$ and $y \in Y$.
For $\Gamma_1 = \Gamma[X]$ and $\Gamma_2 = \Gamma[Y]$ we say that \emph{$\Gamma$ is the join of $\Gamma_1$ and $\Gamma_2$} and we write $\Gamma = \Gamma_1 \vee \Gamma_2$.
We call an edge $uv \in V(\Gamma_1 \vee \Gamma_2)$ a \emph{joining edge} when $u \in V(\Gamma_1)$ and $v \in V(\Gamma_2)$.
Notice that if $\Gamma$ is a join, then $V(\Gamma)$ admits a partition: $\{X_1,X_2,\dots, X_k\}$ such that $\Gamma[X_i]$ is not a join for each $1 \leq i \leq k$ and for each $1 \leq i < j \leq k$ we have $\Gamma[X_i \cup X_j] = \Gamma[X_i] \vee \Gamma[X_j]$.
We denote such a decomposition as $\Gamma= \bigvee_{1 \leq i \leq k} \Gamma[X_i]$.

\begin{lemma}\label{lem:joinsWork}
	Let $\Gamma_1$ and $\Gamma_2$ be graphs such that each of $\Gamma_1$ and $\Gamma_2$  have no isolated vertices.
	The graph $\Gamma_1 \vee \Gamma_2$ admits a non-trivial  chromatically invariant $2$-edge colouring in which every vertex is incident with at least one red edge and one blue edge.
\end{lemma}

\begin{proof}
	Let $J$ be the set of joining edges of $\Gamma_1 \vee \Gamma_2$.
	Let $R = E(\Gamma_1) \cup E(\Gamma_2)$ and $B = J$.
	Since each of $\Gamma_1$ and $\Gamma_2$  have no isolated vertices, each vertex of $\Gamma_1 \vee \Gamma_2$ is incident with at least one red edge and one blue edge.
	For any pair of disjoint independent sets $I_1,I_2 \subset V(\Gamma)$ we have $I_1,I_2 \subset V(\Gamma_1)$ or $I_1,I_2 \subset V(\Gamma_2)$.
	The result follows by observing that the $2$-edge-coloured graph $(\Gamma,R,B)$ satisfies the hypothesis of Theorem \ref{thm:I1I2}
\end{proof}

\begin{lemma}\label{lem:monochromeJoin}
	Let $G=\left(\Gamma,R,B\right)$ be  non-trivial chromatically invariant $2$-edge-coloured graph. If there exists graphs $\Gamma_1$ and $\Gamma_2$  so that $\Gamma = \Gamma_1 \vee \Gamma_2$ and neither of $\Gamma_1$ or $\Gamma_2$ is a join, then all of the joining edges of $\Gamma_1 \vee \Gamma_2$ have the same colour in $G$.
\end{lemma}

\begin{proof}
	Let $\Gamma_1$ and $\Gamma_2$ be graphs that are not joins.
	Let $\Gamma = \Gamma_1 \vee \Gamma_2$.
	Let $G=\left(\Gamma,R,B\right)$ be  non-trivial chromatically invariant $2$-edge-coloured graph.
	
	We first show that for any pair $y_1,y_2 \in V(\Gamma_2)$ there is a sequence of independent sets $I_1,I_2,\dots, I_\ell$ so that $y_1 \in I_1$, $y_2 \in I_\ell$ and $I_i \cap I_{i+1} \neq \emptyset$ for all $1 \leq i \leq \ell-1$.
	Since $\Gamma_2$ is not a join, its complement, $\overline{\Gamma_2}$, is connected.
	Therefore there is a path from $y_1$ to $y_2$ in $\overline{\Gamma_2}$.
	The edges of such a path form a sequence of independent sets in $\Gamma_2$: $I_1,I_2,\dots, I_\ell$ so that $y_1 \in I_1$, $y_2 \in I_\ell$ and $I_i \cap I_{i+1} \neq \emptyset$ for all $1 \leq i \leq \ell$.
	
	Consider $v \in V(\Gamma_1)$ and $y_1,y_2 \in V(\Gamma_2)$.
	Since $\Gamma = \Gamma_1 \vee \Gamma_2$, we have $vy_1,vy_2 \in E(\Gamma)$.
	Let $I_1,I_2,\dots, I_\ell$ be a sequence of independent sets in so that $y_1 \in I_1$, $y_2 \in I_\ell$ and $I_i \cap I_{i+1} \neq \emptyset$ for all $1 \leq i \leq\ell-1$.	
	By Lemma \ref{thm:I1I2} edges between $I_i$ and $v$ all have the same colour.
	Since $I_i \cap I_{i+1} \neq \emptyset$, all of the edges between $I_{i+1}$ and $v$ all have that same colour.
	Since $v$ is adjacent to every vertex in $\Gamma_2$ it follows that the edges between $I_{1}$ and $v$ have the same colour as those between $I_{\ell}$ and $v$.
	Therefore every joining edge with an end at $v$ has the same colour.
	
	Similarly, for any $u \in  V(\Gamma_2)$, every joining edge with an end at $u$ has the same colour in $G$.
	Thus it follows that  all of the joining edges of $\Gamma_1 \vee \Gamma_2$ have the same colour in $G$.	
\end{proof}

\begin{lemma}\label{lem:isaJoin}
	If $G = \left(\Gamma,R,B\right)$ is a chromatically invariant $2$-edge-coloured graph in which every vertex is incident with a red edge and a blue edge, then $\Gamma$ is a join.
\end{lemma}

\begin{proof}
	Let $G = (\Gamma,R,B)$ be a minimum counterexample with respect to number of vertices.
	We first show that there exists a vertex $v \in V(\Gamma)$ so that every vertex of $G-v$ is incident with both a red edge and a blue edge.
	
	If no such vertex exists, then for every $x \in V(\Gamma)$ there exists $y \in V(\Gamma)$ so that $xy \in E(\Gamma)$ and the edge $xy$ is the only one of its colour incident with $y$.
	We proceed in cases based on the existence of a pair $u,v \in V(\Gamma)$ so that $uv$ is the only one of its colour incident with $u$ and the only one of its colour incident with $v$.
	
	If such a pair exists, then, without loss of generality, let $uv$ be red.
	Since $G$ has no induced bichromatic $2$-path, $N_{G[B]}(u) = N_{G[B]}(v)$.
	If $V(\Gamma) = \{u,v\} \cup N_{G[B]}(u)$, then $\Gamma$ is a join.
	As such the set $Q = V(\Gamma) \setminus \left(\{u,v\} \cup N_{G[B]}(u)\right)$ is non-empty.
	Notice that as $G$ has no induced bichromatic $2$-path, all edges between $Q$ and $N_{G[B]}(u)$ are blue.
	Since every vertex of $G$ is incident with both a red and a blue edge, it follows that every vertex of $Q$ is incident with a red edge in the $2$-edge-coloured graph $G[Q]$.
	As $\Gamma$ is not a join, there exists $q \in Q$ and $x \in   N_{G[B]}(u)$ so that $qx \notin E(\Gamma)$.
	Let $rq$ be a red edge in $G[Q]$. 
	Notice $rx \notin E(\Gamma)$, as otherwise such an edge is blue in $G$ and so $xrq$ is an induced bichromatic $2$-path in $G$.
	Therefore the subgraph induced by $\{u,x,q,r\}$ is a bichromatic copy of $2K_2$.
	This is a contradiction as $G$ is chromatically invariant.
	Therefore no such pair $u,v$ exists.
	
	Since no such pair $u,v$ exists, there is a maximal sequence of  vertices of $\Gamma$: $u_1,u_2,\dots u_k$ so that  $u_iu_{i+1} \in E(\Gamma)$ and the edge $u_iu_{i+1}$ is the only one of its colour incident with $u_{i+1}$ for all $1 \leq i \leq k-1$.
	We further note that, without loss of generality, vertices with an even index are adjacent with a single red edge and vertices with an odd index (other that $u_1$) are incident with a single blue edge.
	Thus $u_1,u_2,\dots,u_k$ is an path whose edges alternate being red and blue.
	If $k \geq 4$, then since $G$ has no induced bichromatic $2$-path, we have $u_2u_4 \in E(\Gamma)$. 
	However, this edge is either a second red edge incident with $u_2$ or a second blue edge incident with $u_4$. This is a contradiction, and so $k = 3$.
	
	Since this path was chosen to be maximal, it follows that the edge between $u_1$ and $u_3$ is the only one of its colour incident with $u_1$.
	Since $u_1u_2$ is red, it follows that $u_1u_3$ is blue.
	But then $u_3$ is incident with two blue edges: $u_2u_3$ and $u_1u_3$.
	This is a contradiction.
	And so there exists a vertex $v \in V(\Gamma)$ so that every vertex of $G-v$ is incident with both a red edge and a blue edge.
	
	Consider $v \in V(\Gamma)$ so that every vertex of $G-v$ is incident with both a red edge and a blue edge.
	By Lemma \ref{lem:monochromeJoin}, $G-v$ is a  chromatically invariant $2$-edge-coloured  graph.
	By the minimality of $G$, we have that $\left(\Gamma - v \right)$ is  a join.
	Therefore $V(\Gamma-v)$ admits a partition  $\{X_1,X_2,\dots, X_k\}$ such that $\left(\Gamma - v \right)[X_i]$ is not a join for each $1 \leq i \leq k$ and for each $1 \leq i < j \leq k$ we have $\left(\Gamma - v \right)[X_i \cup X_j] = \left(\Gamma - v \right)[X_i] \vee \left(\Gamma - v \right)[X_j]$. 
	
	Notice for any $1 \leq i \leq k$ that if $v$ is adjacent to every vertex of $X_i$, then $\Gamma$ is necessarily a join.
	Thus, for every for every $1 \leq i \leq k$ vertex $v$ is not adjacent to at least one vertex of $X_i$.
	
	By hypothesis, $v$ in incident with a red edge and a blue edge.
	Let $vr$ and $vb$ be such edges for some $r,b \in V(\Gamma)$.
	We proceed in cases based on the location of $r$ and $b$ within the partition $\{X_1,X_2,\dots, X_k\}$ of $V(\Gamma-v)$.
	
	Consider, without loss of generality, $r,b \in X_1$.
	Notice that by Theorem \ref{cor:subgraphs} and  Lemma \ref{lem:monochromeJoin}, for every $1 \leq i < j \leq k$, the joining edges of $\left(\Gamma - v \right)[X_i \cup X_j] = \left(\Gamma - v \right)[X_i] \vee \left(\Gamma - v \right)[X_j]$ are the same colour.
	Since each of $r$ and $b$ are adjacent to every vertex of $X_2$, then either $vry$ or $vby$ is an induced bichromatic $2$-path for every $y \in X_2$.
	(Whether or not  $vry$ or $vby$ is an induced bichromatic $2$-path depends on the colour of the joining edges between $X_1$ and $X_2$ ).
	Since $G$ is chromatically invariant, by Theorem \ref{thm:firstCharact}	no such path can be exist 
	And so $v$ is adjacent to every vertex in $X_2$, a contradiction.
	
	Consider, without loss of generality, $r \in X_1$ and $b \in X_2$.
	Further assume without loss of generality, that all of the joining edges of $\left(\Gamma - v \right)[X_1 \cup X_2] = \left(\Gamma - v \right)[X_1] \vee \left(\Gamma - v \right)[X_2]$ are red.
	Since $\Gamma$ is not a join, there is at least one vertex of $X_1$, say $x_1$, that is not adjacent to $v$.
	However, $x_1bv$ is an induced bichromatic $2$-path.
	Since $G$ is chromatically invariant, by Theorem \ref{thm:firstCharact}	no such path can be exist.
	This is a contradiction.
\end{proof}

Together Lemmas \ref{lem:joinsWork} and \ref{lem:isaJoin} imply the following characterization of those graphs which admit non-trivial chromatically invariant $2$-edge-colourings in which every vertex is incident with at least one edge of both colours.

\begin{theorem}\label{thm:fullJoin}
	A graph $\Gamma$ admits a non-trivial chromatically invariant $2$-edge-colouring in which every vertex is incident with both a red edge and a blue edge if and only if $\Gamma$ is the join of two graphs, each of which has no isolated vertices.
\end{theorem}

In \cite{C19} the authors fully characterize those oriented graphs that admit a chromatically invariant orientation. 
Theorem \ref{thm:fullJoin} presents a partial analogue for $2$-edge-coloured graphs.
It remains unknown which graphs admit a non-trivial chromatically invariant $2$-edge-colouring when we allow for vertices to be incident with edges of a single colour.

\section{Roots of Chromatic Polynomials of $2$-edge-coloured Graphs}\label{sec:roots}

As with many graph polynomials it is natural to study the location of the roots of the chromatic polynomial of $2$-edge-coloured graphs. The chromatic polynomial of graphs is well studied, and one topic of interest is the location of the roots. A root of a chromatic polynomial of a graph is called a \emph{chromatic root}. See Figure~\ref{figchromroots6} for the chromatic roots of all connected graphs on 6 vertices obtained by computer search. We call a root of a chromatic polynomial of a $2$-edge-coloured graph a \emph{monochromatic root} if the graph is monochromatic and a \emph{bichromatic root} if the graph is bichromatic. See Figure~\ref{figroots6} for the bichromatic roots of all connected $2$-edge-coloured graphs on 6 vertices obtained by computer search. Recall the chromatic polynomial of a monochromatic graph is simply the chromatic polynomial of the underlying graph. Therefore the collection of all chromatic roots is exactly the collection of all monochromatic roots. In this section we provide results on bichromatic roots. 

\begin{figure}[h]~\label{figroots}
\centering
\begin{minipage}{.5\textwidth}
  \centering
	\includegraphics[scale=0.25]{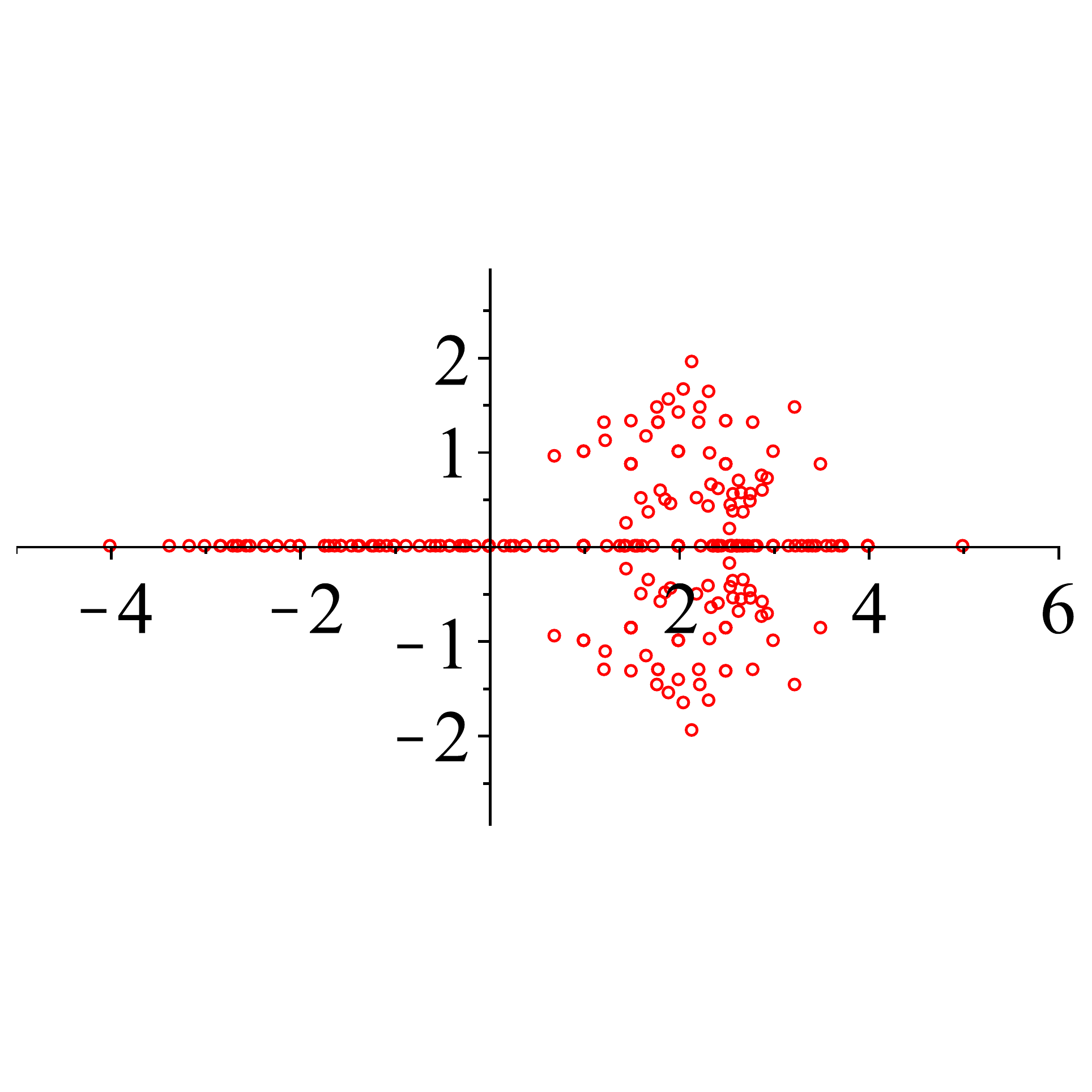}
	\caption{Bichromatic roots of all connected $2$-edge-coloured graphs on 6 vertices}
  \label{figroots6}
  \end{minipage}%
\begin{minipage}{.5\textwidth}
  \centering
	\includegraphics[scale=0.25]{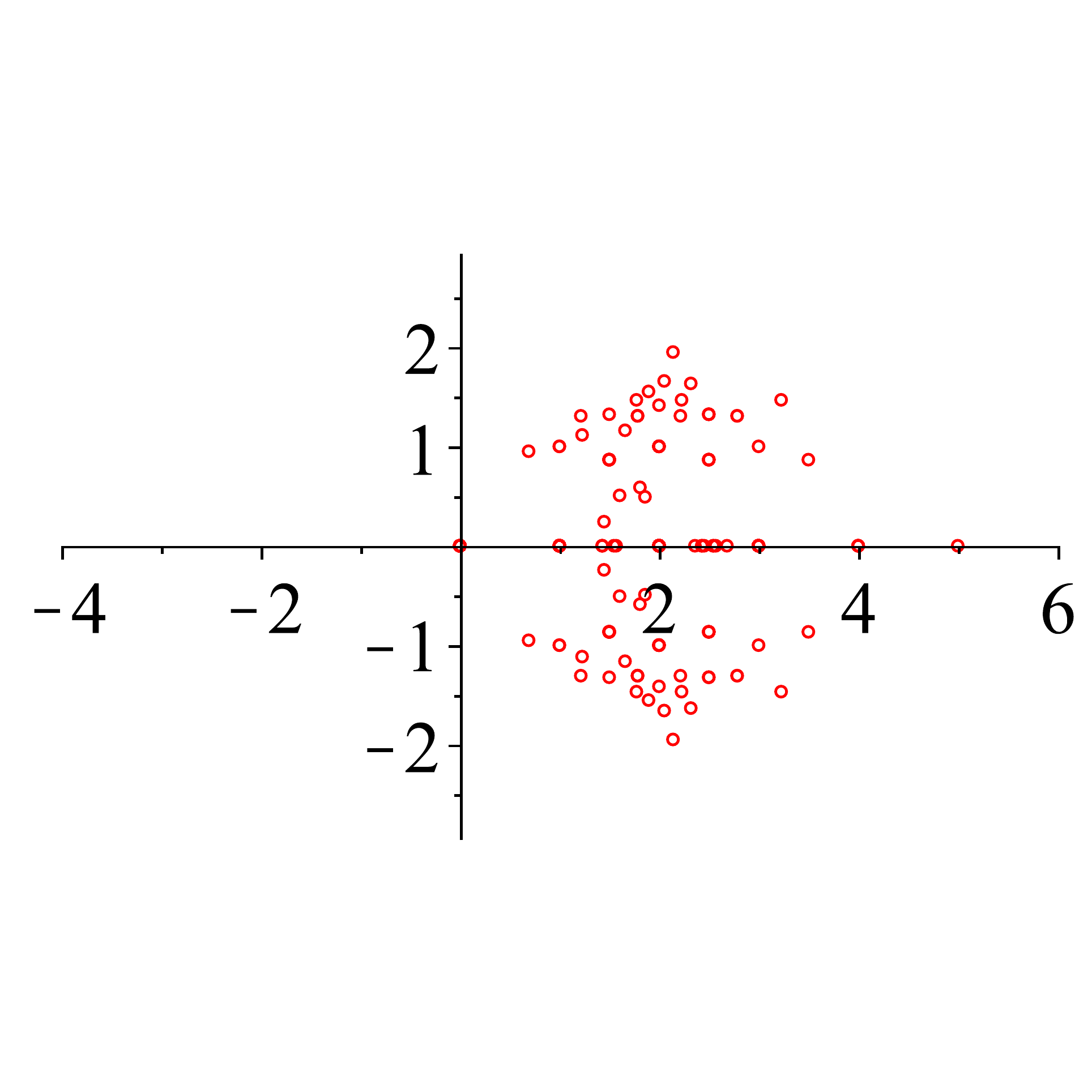}
	\caption{Chromatic roots of all connected graphs on 6 vertices}
  \label{figchromroots6}
  \end{minipage}

\end{figure}

We begin with a study of the real roots. The real chromatic roots are always positive \cite{R68} as the coefficients of the chromatic polynomial of a graph alternate in sign and there are no real roots in $(0,1)\cup (1,\frac{32}{27}]$. In contrast we will show the bichromtic roots are dense in $\mathbb{R}$ and the collection of all rational bichromatic roots is $\mathbb{Z}$.
For $n > 1$, let $K_n^r$ and $K_n^b$  denote monochromatic copies of $K_n$ with red and blue edges, respectively. 
For a pair of $2$-edge-coloured graphs $G= \left(\Gamma_G,R_G, B_G\right)$ and $H = \left(\Gamma_H, R_H, B_H\right)$, let $G \cup H$ denote the disjoint union of $G$ and $H$.
%That is, $G \cup H =\left(\Gamma_G \cup \Gamma_H, R_G \cup R_H, B_G \cup B_H\right)$

\begin{theorem}
	\label{thm:GK2}
	Let $G=(\Gamma,R,B)$ be a $2$-edge-coloured graph on $n$ vertices so that $\chi(G)=n$. 
	We have
	
	\[P(G \cup K_2^r, \lambda)=\lambda(\lambda-1) \cdots (\lambda-n+1)(\lambda^2-\lambda-2|B|).\]
\end{theorem}

\begin{proof}
	Let $G=(\Gamma,R,B)$ be a $2$-edge-coloured graph on $n$ vertices so that $\chi(G)=n$. 
	For fixed $k > 0$, we construct a $k$-colouring of $G \cup K_2^r$ by first colouring vertices $G$ and then those of $K_2^r$. 
	As $\chi(G)=n$, any $k$-colouring assigns each vertex of $G$ a unique colour.
	There are $k(k-1) \cdots (k-|V|+1)$ such colourings. 
	We can then colour vertices of $K_2^r$ with any two different colours unless there exists $b \in B$ whose ends have been assigned those two colours.
	Thus each $b \in B$ prohibits two possible colourings of the $K_2^r$. 
	Therefore given any $k$-colouring of $G$ there are $k^2-k-2|B|$ such colourings of $K_2^r$.
	And so \[P(G \cup K_2^r, \lambda)=\lambda(\lambda-1) \cdots (\lambda-n+1)(\lambda^2-\lambda-2|B|).\]
\end{proof}

\begin{corollary}
	\label{cor:introot}
	Let $n > 1$ be an integer.
	We have 
	\[P(K_n^b \cup K_2^r, \lambda)=\lambda(\lambda-1) \cdots (\lambda-n)(\lambda+n-1).\]
\end{corollary}

\begin{theorem}
	The closure of the rational bichromatic roots is $\mathbb{Z}$.
\end{theorem}

\begin{proof}
	By Theorem \ref{thm:basicResults}, for any $2$-edge-coloured graph $G$, the leading coefficient of $P(G,\lambda)$ is 1. 
	Therefore by the rational root theorem, any rational root of an oriented chromatic polynomial must be an integer. 
	Let $m > 0$ be an integer
	By Corollary \ref{cor:introot}, we have \[ P(K_{m+1}^b \cup K_2^r,\lambda)= \lambda(\lambda-1) \cdots (\lambda-(m+1))(\lambda+m).\]
	We observe $P(K_{m+1}^b \cup K_r^2,\lambda)$ has roots at $\lambda = -m,0,1,\ldots, m,m+1$.
\end{proof}

\begin{theorem}
	The closure of the real bichromatic roots is $\mathbb{R}$.
\end{theorem}

\begin{proof}
	For any $r \in \mathbb{R}$ let $d(r)=r-\lfloor r \rfloor$. 
	Furthermore let 
	
	\[A = \left\{ \frac{1-\sqrt{1+8m}}{2}: m \in \mathbb{N} \right\}.\]

	Let $G$ be a $2$-edge-colouring of a complete graph so that $|B|=m$.
	By Theorem \ref{thm:GK2} each element of $A$ is a real root of $P(G \cup K_2^r, \lambda)$.
	We first show that for any $r \in \mathbb{R}$ and $\epsilon >0$, there exists an $a \in A$ such that $|d(r)-d(a)|< \epsilon$. 
	
	Let $f(m)= \frac{1-\sqrt{1+8m}}{2}$ and $M=\frac{2}{\epsilon^2}$. 
	Note that $|f(m+1)-f(m)| < \epsilon$ for $m \geq M$. 
	Furthermore $f(m) \rightarrow -\infty$ as $m\rightarrow \infty$. 
	Thus for any $s \leq f(M)$, there exists an $m \geq M$ such that $f(m+1) \leq s \leq f(m) \leq r$.
	This implies $|s-f(m)| < \epsilon$. 
	Furthermore $|d(s)-d(f(m))| < \epsilon$. 
	By choosing $s \leq f(M)$ such that $d(s)=d(r)$ and let $a=f(m) \in A$.
	It then follows that $|d(r)-d(a)|< \epsilon$. 
	
	Let $G_a$ be a $2$-edge-coloured graph so that $P(G_a,\lambda)$ has a real root at $a$.
	Let $H_a$ be the $2$-edge-coloured graph formed the join of $G_a$ and any $2$-edge-coloured $K_n$, where all of the joining edges are red.
	We have 
	
	\[P(H_a,\lambda)=\lambda(\lambda-1) \cdots (\lambda-n+1)P(G_a,\lambda-n).\]

	Consider $n=\lfloor r \rfloor- \lfloor a \rfloor$.
	As $r=d(r)+\lfloor r \rfloor$ and $a+n=d(a)+\lfloor a \rfloor+n = d(a)+\lfloor r \rfloor$ we have 
	\[|r-(a+n)| = |d(r)+\lfloor r \rfloor-(d(a)+\lfloor r \rfloor)|=|d(r)-d(a)|< \epsilon.\]	
	Thus $H_a$ has a root at $a+n$.

\end{proof}

We turn now to study complex roots of chromatic polynomials of $2$-edge-coloured graphs.
We show they may have arbitrarily large modulus. 
To do this we study the limit of the complex roots of a $2$-edge-coloured complete bipartite graph.  

Let $p_n(z)=\sum_{j=1}^k\alpha_j(z)\lambda_j(z)^n$. 
Beraha, Kahne and Weiss studied the limits of the complex roots of such functions (as arising in recurrences).
They fully classified those values that occur as limits of roots of a family of polynomials.
See \cite{W78} for a full statement of the Bereha-Kahane-Weiss Theorem.

A limit of roots of a family of polynomials ${P_n}$ is a complex number, $z$, for which there are sequences of integers $(n_k)$ and complex numbers $(z_k)$ such that $z_k$ is a zero of $P_{n_k}$, and $z_k \to z$ as $k \to \infty$. 
The Bereha-Kahane-Weiss Theorem requires non-degeneracy conditions: no $\alpha_i$ is identically $0$, and $\lambda_i\not=\omega\lambda_k$  for any $i\not = k$ and any root of unity $\omega$. 
The Bereha-Kahane-Weiss Theorem implies that the limit of roots of $P_n(z)$ are precisely those complex numbers $z$ such that one of the following hold:

\begin{itemize}
	\item one of the $|\lambda_i(z)|$ exceeds all others and $\alpha_i(z)=0$; or
	\item $|\lambda_1(z)|=|\lambda_2(z)|=...=\lambda_{\ell}(z)>|\lambda_j(z)|$ for $\ell + 1\leq j \leq k$ for some $\ell \geq 2$. 
\end{itemize}

\begin{theorem}~\label{Thrm:complexroots}
	Non-real bichromatic roots can have arbitrarily large modulus.
\end{theorem}

\begin{proof}	
	Let $n \geq 5$ be an integer.
 	Consider $K_{2,n-2}$ with partition $\{X,Y\}$ with $X = \{u,v\}$.
	Let $G = \left(K_{2,n-2},R,B\right)$ so that three of the edges incident with $v$ are blue and all other edges are red.
	Let $x,y,z\in Y$ be the vertices of $G$ that are adjacent to $v$ by a blue edge.
	In any $k$-colouring $c$, we have $c(u)\neq c(v)$.
	Further, for each $w \in Y \setminus\{x,y,z\}$ we must have $c(w) \neq c(x),c(y),c(z)$.
	We proceed to count the number of $k$-colourings of $G$ based on the cardinality of $\left|\left\{ c(x),c(y),c(z)\right\}\right|$.
	
	 When $\left|\left\{ c(x),c(y),c(z)\right\}\right| = 3$, there are 
	 \[
	 k(k-1)(k-2)(k-3)(k-4)(k-5)^{n-5}
	 \]
	 $k$-colourings of $G$.
	 When  $\left|\left\{ c(x),c(y),c(z)\right\}\right| = 2$, there are 
	 \[
	 3k(k-1)(k-2)(k-3)(k-4)^{n-5}
	 \]
	 $k$-colourings of $G$.
	 Finally, when $\left|\left\{ c(x),c(y),c(z)\right\}\right| = 1$, there are 
	 \[
	 k(k-1)(k-2)(k-3)^{n-5}
	 \]
	 $k$-colourings of $G$.
	Thus
	
		\begin{align*}
	P(G,\lambda) &= (\lambda-2)(\lambda-3)(\lambda-4)(\lambda-5)^{n-5} + 3(\lambda-2)(\lambda-3)(\lambda-4)^{n-5} + \lambda(\lambda-1)(\lambda-2)(\lambda-3)^{n-5}\\
	&= \lambda(\lambda-1)(\lambda-2)(\lambda-3)\left((\lambda-3)^{n-6}+3(\lambda-4)^{n-5}+(\lambda-4)(\lambda-5)^{n-5}\right)
	\end{align*}

	Consider the polynomial $g(n,\lambda) = (\lambda-3)^{n-6}+3(\lambda-4)^{n-5}+(\lambda-4)(\lambda-5)^{n-5}$.
	We may express this polynomial as:
	
	\[
	p(n,z)=\alpha_1(z)(\lambda_1(z))^{n-6}+\alpha_2(z)(\lambda_2(z))^{n-5}+\alpha_3(z)(\lambda_3(z))^{n-5}.
	\]

	Here the non-degeneracy conditions hold for $p(n,z)$. 
	Applying the Bereha-Kahne-Weiss Theorem and setting $|\lambda_1(z)|=|\lambda_3(z)|>|\lambda_4(z)|$ we solve for  $z=a+bi$ such that $|z-3|=|z-5|>|z-4|$.
	One can verify when $a=4$  we have $|z-3|=|z-5|$ and $|z-5|>|z-4|$ for all values of $b$.
	Thus the curve $z=4+bi$ is a limit of the roots for $2$-edge-coloured chromatic polynomial of $K_{2,n-2}$.
	As there are no restrictions on $b$, it then follows that $P(G,\lambda)$ can have complex roots of arbitrarily large modulus.
\end{proof}

Consider our above $K_{2,\ell-2}$ with partition $\{X,Y\}$ with $X = \{u,v\}$ and the 2-edge-coloured graph $G_{\ell} = \left(K_{2,\ell-2},R,B\right)$ with three of the edges incident with $v$ are blue and all other edges are red. 
Let $H_{n,\ell}$ be the $2$-edge-coloured graph formed by $G_{\ell}$ by joining $G_{\ell}$ with a copy of $K^r_n$ so that all joining edges are blue. 
Every vertex of  $K^r_n$ requires a distinct colour and the joining edges are all blue.
Therefore no vertex of $G$ can be assigned any of the $n$ colours appearing on the vertices of  $K^r_n$. 
Thus

\[ P(H_{n,\ell},\lambda)=\lambda(\lambda-1) \cdots (\lambda-(n-1))P(G_{\ell},\lambda -n)\]

Taken with Theorem~\ref{Thrm:complexroots}, this implies that the curve $f(n,b)=4+n+bi$ is also limit of the roots for $n\geq 1$. 
See Figures~\ref{figrootscomplex1} and \ref{figrootscomplex2} for a plot of the roots of these polynomials.

\begin{figure}
\centering
\begin{minipage}{.5\textwidth}
  \centering
  \includegraphics[width=.8\linewidth]{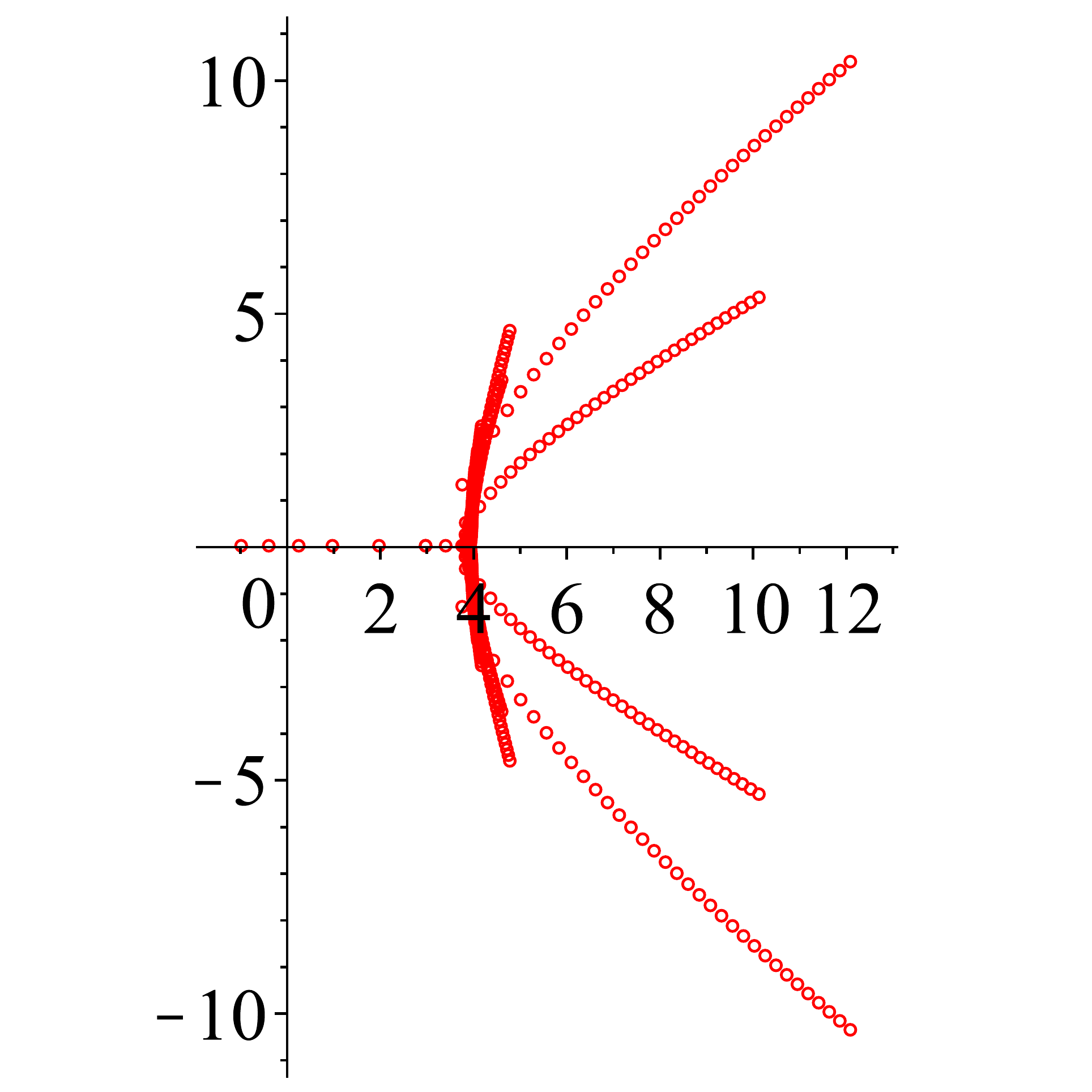}
  \caption{Complex Chromatic Roots of the 2-edge coloured $K_{2,n-2}$}
  \label{figrootscomplex1}
  \end{minipage}%
\begin{minipage}{.5\textwidth}
  \centering
  \includegraphics[width=.8\linewidth]{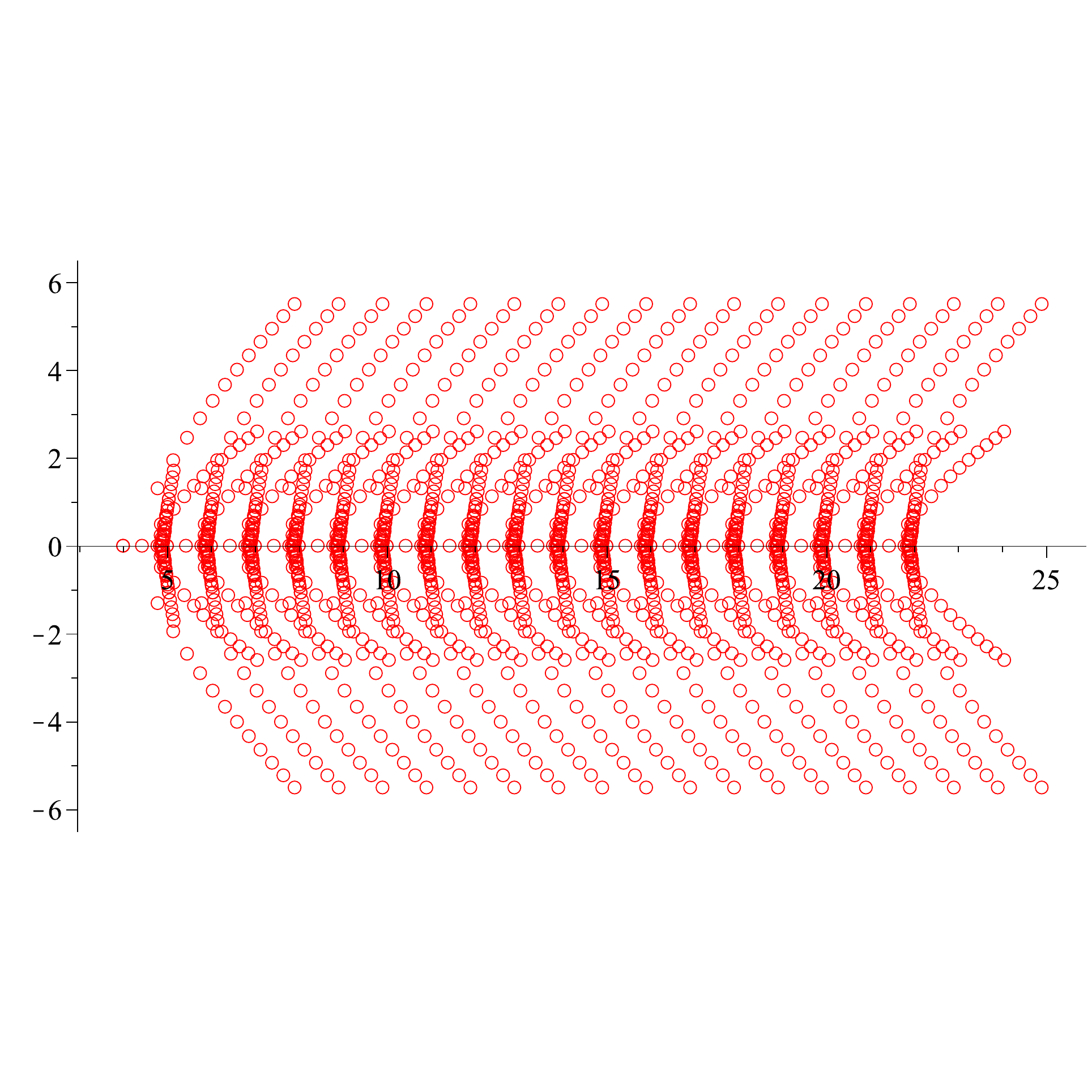}
  \caption{Complex chromatic roots of $H_{n,\ell}$ for $\ell=6,...,18$, $n=1,...,18$}
  \label{figrootscomplex2}
  \end{minipage}

\end{figure}

From the plots in Figure~\ref{figrootscomplex1} and Figure~\ref{figrootscomplex2} one can see that the closure of the roots contain an infinite number of vertical curves crossing the real axis at integer values of at least $4$. The real and complex chromatic roots (and hence monochromaitc roots) are dense in the complex plane \cite{S04}. It remains to be seen if bichromatic roots that dense in the complex plane.

\section{Further Remarks} \label{sec:discuss}
That the results and methods in Section \ref{sec:chromPoly} closely resemble results and methods for the chromatic polynomial comes as no surprise to these authors -- Equations \ref{math:complete} and \ref{math:reduce} both hold for the chromatic polynomial of a graph.
We note, however, that the standard \emph{delete and contract} technique for the chromatic polynomial of a graph does not apply, in general, for $2$-edge-coloured graphs. 
Deleting an edge, in some sense, forgets the colour of the adjacency between a pair of vertices -- important information for a vertex colouring a $2$-edge-coloured graph.
This implies that the $2$-edge-coloured chromatic polynomial is not an evaluation of the Tutte polynomial.

The results and methods in Section \ref{sec:chromPoly} closely mirror those for the oriented chromatic polynomial in \cite{C19}.
Such a phenomenon has been observed in the study of the chromatic number oriented graphs and $2$-edge-coloured graphs.
In \cite{RASO94} Raspaud and Sopena give an upper bound for the chromatic number of an orientation of a planar graph.
And in \cite{AM98} Alon and Marshall use the same techniques to derive the same upper bound for the chromatic number of a $2$-edge-coloured planar graph.
In this latter work, Alon and Marshall profess the similarity of their techniques to those appearing in \cite{RASO94}, yet see no way to derive one set of results from the other.
In the following years Ne\v{s}et\v{r}il and  Raspaud \cite{NR00} showed these results were in fact special cases of a more general result for  $(m,n)$-mixed  graphs --  graphs in which there are $m$ different arc colours and $n$ different edge colours.
Oriented graphs are graphs are $(0,1)$-mixed graphs, oriented graphs  are $(1,0)$-mixed graphs and $2$-edge-coloured graphs $(0,2)$-mixed graphs.

By way of homomorphism, one can define, for each $(m,n) \neq (0,0)$, colouring for $(m,n)$-mixed graphs that  generalizes graph colouring, oriented graph colouring and $2$-edge-coloured graph colouring.
%And we expect that the chromatic polynomial of an (m,n)-mixed graph will unify results for graphs, oriented graphs and $2$-edge-coloured graphs.
As our results in Section \ref{sec:chromPoly} closely mirror those in \cite{C19} for oriented graphs we expect that the results in Section \ref{sec:chromPoly} in fact special cases of a more general result for the, to be defined, chromatic polynomial of an $(m,n)$-mixed graph.
Showing such a result would require successfully generalizing the notions of obstructing arcs/edge as well as the notions of $2$-dipath and bichromatic $2$-path.
This latter problem is considered in \cite{BDS17}.

%We note however, that though we expect the generalization of the results in Section \ref{sec:chromPoly} to be mostly straightforward, the same cannot be said about the results for chromatic invariance. 
%The analogue of Theorem \ref{thm:fullNotJoin} for oriented graphs is the family of quasi-transitive orientations of co-interval graphs \cite{C19}.
%This result bears little resemblance to the statement of  Theorem \ref{thm:fullNotJoin}.

\bibliographystyle{abbrv}
\bibliography{2ecrefs}

\end{document}